\documentclass[11pt]{amsart}
\usepackage[english]{babel}
\usepackage{inputenc}

\usepackage{latexsym}
\usepackage{amssymb}
\usepackage{amsmath}
\usepackage{amsfonts}
\usepackage[mathscr]{eucal}
\usepackage{xcolor}

\usepackage{amscd}

\usepackage{graphicx}
\usepackage{epsfig}
\usepackage{relsize}
\usepackage{fullpage}

\usepackage{longtable}
\usepackage{rotating}
\usepackage{array}
\usepackage{multicol}
\usepackage{multirow}
\usepackage{makecell}
\usepackage{mathtools}
\usepackage{booktabs}

\usepackage{amsthm}
\usepackage[all,cmtip]{xy}
\usepackage[pdfstartview=FitH]{hyperref}
\usepackage{comment}
\usepackage{psfrag}
\usepackage{verbatim}
\usepackage{hyperref}
\hypersetup{
colorlinks=true,
linkcolor=blue,
allcolors=blue
}

\usepackage{pifont}
\usepackage{caption}
\usepackage{csquotes}

\usepackage{soul}
 \usepackage[normalem]{ulem}

\usepackage{tikz}
\usetikzlibrary{calc,intersections,through,backgrounds,perspective}
\usepackage{pgfplots}
\pgfplotsset{compat=1.15}
\usepackage{mathrsfs}
\usetikzlibrary{arrows}

\usepackage{enumerate}
\usepackage{enumitem}

\usepackage[style=numeric,backend=biber,maxbibnames=99]{biblatex}

\newtheorem{thm}{Theorem}[section]
\newtheorem*{thm*}{Theorem}
\newtheorem*{cor*}{Corollary}	

\newtheorem{lemma}[thm]{Lemma}
\newtheorem{proposition}[thm]{Proposition}

\newtheorem{theorem}{Theorem}
\renewcommand*{\thetheorem}{\Alph{theorem}}

\theoremstyle{definition}
\newtheorem{definition}[thm]{Definition} 	
\newtheorem{rema}[thm]{Remark} 
\newtheorem*{definition*}{Definition}    	

\newtheorem*{ack}{Acknowledgements}

\numberwithin{equation}{section}

\DeclareMathOperator{\dist}{dist}

\DeclareMathOperator{\diam}{diam}

\DeclareMathOperator{\GH}{GH}
\renewcommand{\H}{\mathrm{H}}
\DeclareMathOperator{\dis}{dis}
\DeclareMathOperator{\net}{net}

\DeclareMathOperator{\Lip}{Lip}

\DeclareMathOperator{\Conv}{Conv}

\newcommand{\RR}{\mathbb{R}}

\newcommand{\eps}{\varepsilon}

\newcommand{\RRR}{\mathscr{R}}

\newcommand{\NN}{\mathbb{N}}

\DeclareMathOperator{\opt}{opt}

\newcommand{\toGH}{\xrightarrow[]{\mathsf{GH}}}

\newcommand{\BB}{\overline{B}}
\newcommand{\calA}{\mathcal{A}}
\newcommand{\calM}{\mathcal{M}}

\newcommand{\calS}{\mathcal{S}}

\newcommand{\calB}{\mathcal{B}}
\newcommand{\calX}{\mathcal{X}}
\newcommand{\calY}{\mathcal{Y}}

\newcommand{\dd}{\mathbf{d}}
\newcommand{\tri}{\triangle}

\usepackage{microtype}
\usepackage{xurl}

\begin{document}


\title{Metric Pairs and Tuples in Theory and Applications}


\author[A.~Ahumada Gómez]{Andrés Ahumada Gómez}

\author[M.~Che]{Mauricio Che}

\author[M.~Cuerno]{Manuel Cuerno}


\address[A.~Ahumada Gómez]{Department of Mathematics, CUNEF Universidad, Madrid, Spain}
\email{andres.ahumada@cunef.edu}

\address[M.~Che]{Faculty of Mathematics, University of Vienna, Austria}
\email{mauricio.adrian.che.moguel@univie.ac.at}

\address[M.~Cuerno]{Department of Mathematics, CUNEF Universidad, Madrid, Spain}
\email{manuel.mellado@cunef.edu}


\date{\today}

\subjclass[2020]{30L15, 53C23, 53C20, 55N31}


\begin{abstract}
We present theoretical properties of the space of metric pairs equipped with the Gromov--Hausdorff distance. First, we establish the classical metric separability and the geometric geodesicity of this space. Second, we prove an Arzelà--Ascoli-type theorem for metric pairs. Third, extending a result by Cassorla, we show that the set of pairs consisting of a $2$-dimensional compact Riemannian manifold and a $2$-dimensional submanifold with boundary that can be isometrically embedded in $\mathbb{R}^3$ is dense in the space of compact metric pairs. Finally, to broaden the scope of potential applications, we describe 
scenarios where the Gromov--Hausdorff distance between metric pairs or tuples naturally arises.
\end{abstract}

\setcounter{tocdepth}{1}

\maketitle

\section{Introduction}

The Gromov--Hausdorff distance and convergence are fundamental concepts in metric geometry, used to quantitatively compare and study metric spaces. Indeed, the Gromov--Hausdorff distance was introduced as a way to measure how close two different compact metric spaces are. 
Introduced by Edwards \cite{edwards} and rediscovered by Gromov in a series of papers \cite{gromov_1, gromov_2} (see \cite{gromov_3} for an English translation), this concept is a generalization of the classical Hausdorff distance, which compares closed subsets of the same metric space. 


The Gromov--Hausdorff framework has led to numerous important theoretical contributions. One of them is the celebrated Gromov's compactness theorem \cite{gromov_1}, which has motivated fruitful theories of metric measure spaces that naturally arise as limits of sequences of Riemannian manifolds with lower curvature bounds in the Gromov--Hausdorff topology, such as Alexandrov spaces \cite{Yu_Burago_1992}, Ricci limits \cite{cheeger-colding1,cheeger-colding2,cheeger-colding3} and, more generally, (R)CD spaces \cite{ambrosio-gigli-savare,gigli,lott-villani,sturm1,sturm2}. See \cite{alexander-bishop-warped,galaz_guijarro_alexandrov, guijarro_survey,guijarro-santos,ohta, petrunin} for a glimpse of these developments. The Gromov--Hausdorff framework has also been used to formulate the stability of other geometric conditions \cite{cavallucci_GH, fukaya, kasue_GH}. 

Beyond its theoretical significance, the Gromov--Hausdorff distance has found numerous applications. For example, Topological Data Analysis (TDA) uses it as a foundational tool in proving stability theorems for objects like persistence diagrams (see \cite{chazal2009gromov}; see also \cite{Ballester2026} for a brief introduction to TDA). Moreover, it continues to play a key role in recent work on stability theory and computational methods; see \cite{adams2025hausdorffvsgromovhausdorffdistances,bubenik2018interleavinggromovhausdorffdistance,torrascasas2024stability,wee2025cohomology}.

In \cite{che_et_al_2,che_diagrams}, Che and collaborators introduced a notion of Gromov--Hausdorff convergence for metric pairs, which extends the classical pointed Gromov--Hausdorff one (see, for example, \cite{buragoburagoivanov,herron2016}). Subsequently, Ahumada Gómez and Che established the foundations of this theory for metric pairs and tuples, proving results such as embedding, completeness, and compactness theorems \cite{aagmc}. Additionally, they presented a relative version of Fukaya’s theorem on quotient spaces under Gromov--Hausdorff equivariant convergence \cite{fukaya}, as well as a version of the Grove–Petersen–Wu finiteness theorem \cite{grove-petersen-wu} for stratified spaces.

There are also recent references to the formalism of metric pairs. Alattar defined an equivariant version of the convergence of metric pairs in \cite{alattar2024}. Additionally, Sakovich and Sormani introduced multiple notions of convergence of space-times, one of them referred to as the Future-Developed Intrinsic Distance, where the structure of metric pair arises when one considers space-times endowed with the so-called null distance and the distinguished subsets are Cauchy hypersurfaces \cite{sakovich-sormani2025}.

There are other applied contexts where having a notion of distance between metric pairs and tuples seems to be useful. For instance, in spaces of hypernetworks and graphs, as defined in \cite{needham_original, needham2024}. Also in spaces of simplicial complexes, as studied in \cite{nnadi}. Finally, within the theoretical development of TDA, notions analogous to the Gromov--Hausdorff distance considered in this work have been introduced for spaces of persistence diagrams \cite{torras_rocio_2,torras_rocio_1,torrascasas2024stability}.

In this paper, we continue building the theory of Gromov--Hausdorff convergence and distance for metric pairs and tuples as well as presenting some applications that broaden the scope of this tool. First, after showing an equivalent formulation for the Gromov--Hausdorff distance for metric pairs in the sense of correspondences, we begin proving that this distance is geodesic and separable, as well as a generalization of the Arzelà--Ascoli theorem for metric pairs. Together, these results provide a collection of key structural properties, demonstrating that the space of metric pairs has a well-behaved metric geometry suitable for approximation, convergence, and stability analyses.
 The following theorems present some specific notation that will be introduced in sections \ref{s:preliminaries} and \ref{s:geodesics, separability and arzela-ascoli}. 

\begingroup
\def\thetheorem{\ref{thm:geodesic}}
\begin{theorem}
        The metric space $(\GH_1,d_{\GH})$ of compact metric pairs, endowed with the Gromov--Hausdorff distance for metric pairs (see definition \ref{def:Gromov--Hausdorff distance for metric pairs}), is geodesic. Moreover, let $(X,A),(Y,B)\in \GH_1$, $R\in \RRR^{\opt}((X,A),(Y,B))$ (see definitions \ref{def:pair_correspondence} and \ref{def:optimal correspondence}), and a curve $\gamma_R:[0,1]\to \GH_1$ between $(X,A)$ and $(Y,B)$ such that
        \[
        \gamma_R(0):= (X,A),\quad \gamma_R(1):=(Y,B)\quad\text{and}\quad \gamma_R(t):=(R,R|_{A\times B}, d_{\gamma_R(t)})\quad\text{for}\quad t\in(0,1),
        \]
        where for each $(x,y),(x',y')\in R$ and $t\in (0,1)$, the metric $d_{\gamma_R(t)}$ is given by
        \[
        d_{\gamma_R(t)}((x,y),(x',y')):=(1-t)\,d_X(x,x')+t\,d_Y(y,y').
        \]Then, $\gamma_R$ is a geodesic.
    \end{theorem}
\addtocounter{theorem}{-1}
\endgroup

\begingroup
\def\thetheorem{\ref{thm:separable}}
\begin{theorem}
    The metric space $(\GH_1,d_{\GH})$ is separable.
\end{theorem}
\addtocounter{theorem}{-1}
\endgroup

\begingroup
\def\thetheorem{\ref{thm:arzela}}
\begin{theorem}
Consider compact metric pairs $(X_i,A_i)\toGH (X_\infty,A_\infty)$ and $(Y_i,B_i)\toGH (Y_\infty,B_\infty)$, and equicontinuous relative maps $f_i\colon (X_i,A_i)\to (Y_i,B_i)$, for all $i\in \NN$  (see \ref{def:relative map}). Then there exists a subsequence of $\{f_i\}_{i\in\NN}$ converging to a continuous relative map $f_\infty\colon (X_\infty,A_\infty)\to (Y_\infty,B_\infty)$. Moreover, if $\Lip(f_i)\leq K$ and $\Lip(f_i|_{A_i}) \leq L$ for all $i\in \NN$ then $\Lip(f_\infty)\leq K$ and $\Lip(f_\infty|_{A_\infty})\leq L$.
\end{theorem}
\addtocounter{theorem}{-1}
\endgroup

As an application of this framework, in section \ref{s:cassorla} we prove a version of the main theorem in \cite{cassorla}, stating the approximability of metric pairs by surfaces.

\begingroup
\def\thetheorem{\ref{thm:cassorla}}
\begin{theorem}
If $\calM$ is the set of compact length metric pairs and $\calS$ is the subset of $\calM$ consisting of 2-dimensional Riemannian manifolds which can be isometrically embedded in $\RR^3$ together with a 2-dimensional submanifold with boundary, then $\bar{\calS}=\calM$ in $d_{\GH_1}$.
\end{theorem}
\addtocounter{theorem}{-1}
\endgroup

Finally, to highlight the potential of the the Gromov--Hausdorff distances for metric pairs and tuples, in section \ref{s:applications} we compare the metrics between metric tuples with those that have been defined for spaces of hypernetworks and graphs \cite{needham_original, needham2024}; spaces of simplicial complexes \cite{nnadi}; and spaces of persistence matching diagrams \cite{torras_rocio_2, torras_rocio_1,  torrascasas2024stability}.

\begin{ack}
   The authors would like to thank Professor Christina Sormani, for fruitful conversations,  and colleagues and friends for their invaluable comments at an early stage of this manuscript. A. Ahumada Gómez and M. Cuerno have been supported by the research grant PID2024-158664NB-C2 from Miciu/AEI/10.13039/501100011033/FEDER/UE. M. Cuerno has also been supported by the research grant PID2021-124195NB-C32 from the Ministerio de Econom\'ia y Competitividad de Espa\~{na} (MINECO) and by the project “Charting political ideological landscapes in Europe: Fault lines and opportunities (POL-AXES)” - Programa Primas y Problemas de la Fundación BBVA 2023. M.~Che was funded in part by the Austrian Science Fund (FWF) [Grant DOI: 10.55776/STA32], and in part by the Deutsche Forschungsgemeinschaft (DFG, German Research Foundation) under Germany’s Excellence Strategy – EXC-2047/1 – 390685813. Part of this research was carried out at the Hausdorff Institute of Mathematics in Bonn, during the trimester program “Metric Analysis”. M.~Che wishes to express his gratitude to this institution and the organizers of the program for the stimulating atmosphere and the excellent working conditions.
   
   For open access purposes, the authors have applied a CC BY public copyright license to any author-accepted manuscript version arising from this submission.
\end{ack}

\section{Preliminaries}\label{s:preliminaries}

\subsection{Gromov-Hausdorff distance}
\mbox{}%

\vspace{5pt}

We present the basic definitions of the Hausdorff and Gromov--Hausdorff distance. For details and proofs of the following, we refer the reader to \cite{buragoburagoivanov}.

We begin with the definition of the \textit{Hausdorff distance}.

\begin{definition}
    Let $A$, $B$ be subsets of a metric space $(X,d)$. The \textit{Hausdorff distance} between $A$ and $B$ is defined as \[d_\H^{d}(A,B) := \inf\{\eps>0\colon B\subset B_\eps(A) \text{ and }A\subset B_\eps(B)\},\]where $B_\eps(A) = \{x\in X\colon \dist(x,A)<\eps\}.$

    Moreover, the hyperspace $H(X)$ of $X$, i.e., the space of closed subsets of a metric space $X$, endowed with the Hausdorff distance is a metric space whenever $X$ is compact.
\end{definition}

We can globalize this notion to obtain a distance between metric spaces. This is known as the \textit{Gromov--Hausdorff distance}.

\begin{definition}
    Let $(X,d_X)$ and $(Y, d_Y)$ be metric spaces. The \textit{Gromov--Hausdorff distance} between $X$ and $Y$ is defined as 
    \[
    d_{\GH}(X,Y) := \inf\{d_\H^{d_Z}(\varphi(X),\psi(Y))\},
    \]
    where the infimum runs over all metric spaces $(Z,d_Z)$ and isometric embeddings $\varphi\colon X\to Z$ and $\psi\colon Y\to Z$.
    It can also be defined using \textit{admissible metrics} in the disjoint union $X\sqcup Y$ as 
    \[
    d_{\GH}(X,Y) = \inf\{d_\H^\delta(X,Y)\colon \delta \text{ is and admissible metric on }X\sqcup Y\},
    \] 
    where a metric $\delta$ is admissible if $\delta|_{X\times X} = d_X$ and $\delta|_{Y\times Y}=d_Y$.

    Moreover, the space of isometry classes of compact metric spaces, endowed with the Gromov--Hausdorff distance, is a metric space itself.
\end{definition}

\subsection{Gromov-Hausdorff distance for metric pairs and tuples}
\mbox{}%

\vspace{5pt}

The following notions are taken from \cite{aagmc, che_diagrams}.
\begin{definition}
    A \textit{metric pair} $(X,A)$ is a metric space $X$ with a closed subset $A\subset X$. 
    More generally, a \textit{metric tuple} (or $k$-tuple, to emphasize its length) $(X, X_k,\dots,X_1)$ is formed by a metric space $X$ and a nested sequence of closed subsets 
    \[
    X\supseteq X_k\supseteq X_{k-1}\supseteq\dots\supseteq X_1.
    \]
    Observe that a metric $1$-tuple is the same thing as a metric pair.
\end{definition}

\begin{definition}
     Let $(Z,\delta)$ be a metric space, $X,Y \subset Z$ subsets, and $A \subset X$, $B \subset Y$ non-empty closed subsets. The \textit{Hausdorff distance} between $(X,A)$ and $(Y,B)$ is given by 
     \[
     d^\delta_{\H} ((X,A),(Y,B)) := d^\delta_\H (X,Y) + d^\delta_\H(A,B)
     \]
    More generally, if $(Z,\delta)$ is a metric space, $X,Y\subset Z$, and $(X,X_k,\dots,X_1)$ and $(Y,Y_k,\dots,Y_1)$ are metric tuples, the \textit{Hausdorff distance} between them is defined as \[d_\H^{\delta}((X,X_k,\dots,X_1),(Y,Y_k,\dots,Y_1)):=d_\H^{\delta}(X,Y)+\sum_{i=1}^kd_\H^\delta(X_i,Y_i).\]
\end{definition}

In the following definition, and in the sequel, we say that a metric pair $(X,A)$ is compact whenever the space $X$ itself is compact, which of course implies that $A$ is compact as well.

\begin{definition}\label{def:Gromov--Hausdorff distance for metric pairs}
The \textit{Gromov--Hausdorff distance} between two compact metric pairs $(X,A)$ and $(Y,B)$ is defined as
\[
d_{\GH} ((X,A),(Y,B)) := \inf\{d^\delta_\H ((X,A),(Y,B)) : \delta\ \text{admissible on}\ X \sqcup Y\}.
\]    
More generally, the \textit{Gromov--Hausdorff distance} between two metric tuples $(X,X_k,\dots,X_1)$ and $(Y,Y_k,\dots,Y_1)$ is defined as 
\[
d_{\GH}((X,X_k,\dots,X_1),(Y,Y_k,\dots,Y_1)):=\inf\{d_\H^\delta((X,X_k,\dots,X_1),(Y,Y_k,\dots,Y_1))\},
\]
where the infimum is taken over all admissible metrics $\delta$ in $X\sqcup Y$.

We denote by $(\GH_k,d_{\GH})$ the metric space of (isometry classes of) compact metric $k$-tuples (see theorem 2.13 of \cite{aagmc}).
\end{definition}

\subsection{Equivalent formulation}\mbox{}%

\vspace{5pt}

We now present a reformulation of the Gromov--Hausdorff distance for metric pairs by correspondences and distortion, akin to an analogous reformulation of the classical Gromov--Hausdorff distance. Throughout this subsection, $(X,A)$ and $(Y,B)$ will be compact metric pairs. 

\begin{definition}\label{def:pair_correspondence}
    We say that $R\subset X\times Y$ is a \textit{pair correspondence} between $(X,A)$ and $(Y,B)$ if the following conditions hold:    
    \begin{itemize}
    \item for every point $x\in X$ there exists a point $y \in Y$ such that $(x,y)\in R$ and, in particular, for every point $a\in A$ there exists $b\in B$ such that $(a,b)\in R$,
    \item for every point $y\in Y$ there exists a point $x \in X$ such that $(x,y)\in R$ and, in particular, for every point $b\in B$ there exists $a\in A$ such that $(a,b)\in R$. 
    \end{itemize}
    In other words, $R$ is a correspondence between $X$ and $Y$ such that the restriction $R|_{A\times B} := R\cap A\times B$ is a correspondence between $A$ and $B$, in the usual sense.
    
    The \textit{distortion} of $R$ is defined by
    \begin{eqnarray*}
    \dis(R)&=&
    \frac{1}{2}\left(\sup\left\{ \left|d_{X}(x,x')-d_{Y}(y,y')\right|\colon (x,y),(x',y')\in R\right\} \right.\\& & +\left.\sup\left\{ \left|d_{X}(a,a')-d_{Y}(b,b')\right|\colon (a,b),(a',b')\in R|_{A\times B}\right\} \right).
    \end{eqnarray*}
    Observe that any pair correspondence is, in particular, a correspondence in the usual sense between the corresponding metric spaces. Whenever we refer to the classic distortion of a correspondence $R$ between metric spaces, we use the notation $\overline{\dis}(R)$.
\end{definition}

    \begin{thm}\label{teo1}
        For any two metric pairs $(X,A)$ and $(Y,B)$,
        \[
        d_{\GH} ((X,A),(Y,B))=\frac{1}{2}\inf_{R}\dis(R).
        \]
        In other words, $d_{\GH} ((X,A),(Y,B))$ is equal to the infimum of $r>0$ for which there exists a pair correspondence between $(X,A)$ and $(Y,B)$ with $\dis(R)<2r$.
    \end{thm}
\begin{proof}
    First we prove that for any $r>d_{\GH} ((X,A),(Y,B))$, there is a pair correspondence $R$ with $\dis(R)<2r$. We define 
    \[
    R:=\left\{ (x,y)\colon x\in X,\, y\in Y,\, \delta(x,y)<r  \right\}
    \]
    where $\delta$ is an admissible metric on $X\sqcup Y$. This is a correspondence since $\delta(X,Y)+\delta(A,B)<r$. If $(x,y)\in R$ and $(x',y')\in R$, then
    \[
    \delta(x,x')\leq \delta(x,y)+\delta(y,y')+\delta(y',x').
    \]
    Such inequality implies 
    \[
    \left| \delta(x,x')-\delta(y,y')\right|\leq \delta(x,y)+\delta(x',y')<2r.
    \]
    Analogously we get 
    \[
    \left| \delta(a,a')-\delta(b,b')\right|\leq \delta(a,b)+\delta(a',b')<2r.
    \]
    Thus, $\dis(R)<2r$.

    Now, we will prove that 
    \[
    d_{\GH} ((X,A),(Y,B))\leq\frac{1}{2}\inf_{R}\dis(R)
    \]
    Let us take an                                         $r$ such that $\dis(R)=2r$, and define $\delta\colon\left( X\bigsqcup Y\right)\times\left( X\bigsqcup Y\right)\to \RR$ an admissible semimetric  by setting
    \[\delta(y,x):=\delta(x,y):=
    \begin{cases}
        d_X(x,y)&\text{if }x\in X,\, y\in X,\\
        d_Y(x,y)&\text{if }x\in Y,\, y\in Y,\\
         \frac{r}{2}+ \inf\left\{ d_X(x,x')+d_Y(y,y'): (x',y')\in R\right\}.
    \end{cases}
\]

Clearly, this is positive and symmetric. To verify the triangle inequality we have two cases.
\begin{enumerate}
    \item $x,y\in X$ and $z\in Y$.
    \begin{align*}
    \delta(x,z)+\delta(z,y)&=\frac{r}{2}+\inf\left\{ d_{X}(x,x')+d_{Y}(z',z):(x',z')\in R \right\}\\
    & +\frac{r}{2}+\inf\left\{ d_{X}(y,y')+d_{Y}(z'',z):(y',z'')\in R \right\}\\
    &= r+\inf\left\{ \begin{array}{c}d_{X}(x,x')+d_{Y}(z',z)\\+d_{X}(y,y')
    +d_{Y}(z'',z)\end{array}:(x',z'),(y',z'')\in R \right\}\\
    &\stackrel{(*)}{\geq} r+\inf\left\{ d_{X}(x,x')+d_{X}(y,y')+d_{Y}(z',z''):(x',z'),(y',z'')\in R \right\}\\
    &\stackrel{(**)}{\geq} 3r+\inf\left\{ d_{X}(x,x')+d_{X}(y,y')+d_{X}(x',y'):(x',z'),(y',z'')\in R \right\}\\
    &\stackrel{(*)}{\geq}3r+\inf\left\{ d_{X}(x,y):(x',z'),(y',z'')\in R \right\}\\
    &\geq d_{X}(x,y)=\delta(x,y).
    \end{align*}
    \item $x,z\in X$ and $y\in Y$.
    \begin{align*}
        \delta(x,z)+\delta(z,y)&=d_{X}(x,z)+\frac{r}{2}+\inf\left\{ d_{X}(z,z')+d_{Y}(y,y'):(z',y')\in R \right\}\\
        &=\frac{r}{2}+\inf\left\{ d_{X}(x,z)+d_{X}(z,z')+d_{Y}(y,y'):(z',y')\in R \right\}\\
        &\stackrel{(*)}{\geq}\frac{r}{2}+\inf\left\{ d_{X}(x,z')+d_{Y}(y,y'):(z',y')\in R \right\}\\
        &=\delta(x,y).
    \end{align*}
\end{enumerate}
    Inequality ($\ast$) is satisfied because of the triangle inequality and
    ($\ast\ast$) is satisfied because $2r=\dis(R)$  and 
    \[
    \alpha:=\sup_{(x,y),(x',y')\in R}\left\{ \left|d_{X}(x,x')-d_{Y}(y,y')\right|\right\} \geq \sup_{(a,b),(a',b')\in R|_{A\times B}}\left\{ \left|d_{X}(a,a')-d_{Y}(b,b')\right|\right\}=:\beta
    \]
    imply that $\alpha\geq 2r$ and $\beta\leq 2r$. Thus, $d_{X}(x,x')+2r\leq d_{Y}(y,y')$ for every $(x,y),(x',y')\in R$.

    Now we have to calculate that $d^\delta_{\H} ((X,A),(Y,B))\leq r$. Since for every $x\in X$ there exists $y\in Y$ such that $(x,y)\in R$, we get that $\delta(x,Y)\leq \frac{r}{2}$ for every $x\in X$ by the definition of $\delta$. Analogously, $\delta(y,X)\leq \frac{r}{2}$, $\delta(a,B)\leq \frac{r}{2}$ and $\delta(b,A)\leq\frac{r}{2}$ for every $y\in Y$, $a\in A$ and $b\in B$. Thus
    \[
     d_{\H}^{\delta}((X,A),(Y,B))\leq r.\qedhere
    \]
\end{proof}
\begin{rema}
    We can also get a reformulation of the metric $k$-tuple distance between $(X,X_k,\dots,X_1)$ and $(Y,Y_k,\dots,Y_1)$ in terms of correspondences by defining the distortion as
    \begin{align*}
        \dis(R)&=\frac{1}{k+1}\left(\sup\left\{ \left|d_{X}(x,x')-d_{Y}(y,y')\right|\colon (x,y),(x',y')\in R\right\} \right.\\&  +\sum_{i=1}^{k}\left.\sup\left\{ \left|d_{X}(a_i,a'_i)-d_{Y}(b_i,b'_i)\right|\colon (a_i,b_i),(a'_i,b'_i)\in R|_{A_i\times B_i}\right\} \right),
    \end{align*}
    where $R$ is a \textit{tuple correspondence} defined in analogy to definition \ref{def:pair_correspondence}.
    
\end{rema}

\section{Geodesics, separability, and Arzel\`a--Ascoli}\label{s:geodesics, separability and arzela-ascoli}
In this section we prove that the space of compact metric pairs $(\GH_1,d_{\GH})$ is geodesic and separable. We also prove a version of the classical Arzel\`a--Ascoli theorem for relative maps between metric pairs.

\subsection{Geodesicity}
\mbox{}%

\vspace{5pt}
\begin{definition}
    If $(X,d_X)$ and $(Y,d_Y)$ are compact metric spaces, we define the \textit{product space of $X$ and $Y$} as the Cartesian product $X\times Y$ with the metric 
    \[
    \dd((x,y),(x',y'))=\max\left\{ d_X(x,x'),d_Y(y,y')\right\}.
    \]
    Recall that the set of closed subsets of $X\times Y$, $H(X\times Y)$, is a compact metric space using Blaschke's theorem with $d_{\H}^{\dd}$, the Hausdorff metric associated with $\dd$. 
\end{definition}

\begin{lemma}[Extension of Lemma 2.1 in \cite{chowmemo}]\label{lema2.1}
    Let $R$ and $S$ be non-empty pair correspondences between $(X,A)$ and $(Y,B)$. Then
    \[
    \left|\dis(R)-\dis(S) \right|\leq d_{\H}^{\dd}(R,S)
    \]
\end{lemma}
\begin{proof}
    Let $\eta> d_{\H}^{\dd}(R,S)$ and $\varepsilon\in (d_{\H}^{\dd}(R,S),\eta)$. Now, we calculate
    \begin{align*}
        \left|\dis(R)-\dis(S) \right|&=\left| \frac{1}{2} \left(\sup_{(x,y),(x',y')\in R}\left|d_X(x,x')-d_Y(y,y') \right|\right.\right.\\
        & \left.+\sup_{(a,b),(a',b')\in R|_{A\times B}}\left|d_X(a,a')-d_Y(b,b') \right|\right)\\
        & -\frac{1}{2}\left( \sup_{(u,v),(u',v')\in S}\left|d_X(u,u')-d_Y(v,v') \right|\right.\\
        &  \left.\left.+\sup_{(c,e),(c',e')\in S|_{A\times B}}\left|d_X(c,c')-d_Y(e,e') \right|\right)\right|\\
        &\leq \frac{1}{2}\sup\limits_{\substack{(x,y),(x',y')\in R\\ (a,b),(a',b')\in R|_{A\times B}\\ (u,v),(u',v')\in S\\ (c,e),(c',e')\in S|_{A\times B}}}\left|\left|d_X(x,x')-d_Y(y,y') \right|+\left|d_X(a,a')-d_Y(b,b') \right|\right.\\
        & \hspace{3cm}\left.- \left|d_X(u,u')-d_Y(v,v') \right|-\left|d_X(c,c')-d_Y(e,e') \right|\right|\\
        &\stackrel{(***)}{\leq} \frac{1}{2}\sup\limits_{\substack{(x,y),(x',y')\in R\\ (a,b),(a',b')\in R|_{A\times B}\\ (u,v),(u',v')\in S\\ (c,e),(c',e')\in S|_{A\times B}}}\left|d_X(x,x')-d_X(u,u') \right|+\left|d_Y(v,v')-d_Y(y,y') \right|\\
        & \hspace{3cm}+ \left|d_X(a,a')-d_X(c,c') \right|+\left|d_Y(e,e')-d_Y(b,b') \right|\\
        &\stackrel{(*)}{\leq} \frac{1}{2} \sup\limits_{\substack{(x,y),(x',y')\in R\\ (a,b),(a',b')\in R|_{A\times B}\\ (u,v),(u',v')\in S\\ (c,e),(c',e')\in S|_{A\times B}}} d_X(x,u)+d_X(x',u')+d_Y(v,y)+d_Y(v',y')\\
        & \hspace{3cm}+ d_X(a,c)+d_X(a',c')+d_Y(b,e)+d_Y(b',e')
        \end{align*}
        \begin{align*}
        &\leq \frac{1}{2} \sup\limits_{\substack{(x,y),(x',y')\in R\\ (a,b),(a',b')\in R|_{A\times B}\\ (u,v),(u',v')\in S\\ (c,e),(c',e')\in S|_{A\times B}}} 2\dd((x,y),(u,v))+2\dd((x',y'),(u',v'))\\
        & \hspace{3cm} +2\dd((a,b),(c,e))+2\dd((a',b'),(c',e'))\\
        &\leq\frac{1}{2}(8\varepsilon)=4\varepsilon<4\eta.
    \end{align*}
    Since $\eta$ was arbitrary, we get the result.
\end{proof}

\begin{definition}\label{def:optimal correspondence}
    We say that a pair correspondence between $(X,A)$ and $(Y,B)$ is \textit{optimal} if 
    \[
    d_{\GH} ((X,A),(Y,B))=\frac{1}{2}\dis(R).
    \]
    We denote by $\RRR^{\opt}((X,A),(Y,B))$ the set of all closed optimal pair correspondences.
\end{definition}

\begin{proposition}\label{prop_exi_opt_corr}
    $\RRR^{\opt}((X,A),(Y,B))\neq\emptyset$ for any compact metric pairs $(X,A)$ and $(Y,B)$.
\end{proposition}

\begin{proof}
    Let us take a sequence of non-negative numbers $\{\varepsilon_n\}_{n\in\NN}$ which converges to $0$. Now, for each $n$, let $A_n$ and $B_n$ be $\varepsilon_n /2$-nets for $A$ and $B$, respectively. We extend them to obtain $X_n$ and $Y_n$, $\varepsilon/2$-nets for $X$ and $Y$. By the definition of $\varepsilon$-net, we have that $d_{\H}^{\delta}((X_n,A_n),(X,A))\leq \varepsilon$ and $d_{\H}^{\delta}((Y_n,B_n),(Y,B))\leq \varepsilon$ for every admissible metric $\delta$. It implies that $(X_n,A_n)\to(X,A)$ and $(Y_n,B_n)\to(Y,B)$.

    Optimal pair correspondences always exist for finite spaces since we construct them, first for the closed subsets and then for the whole spaces. Then, for each $n$, let be $R_n\in \RRR((X_n,A_n),(Y_n,B_n)) $ such that $\dis(R_n)=2 d_{\GH}((X_n,A_n),(Y_n,B_n))$. The sequence $\{R_n\}_{n\in\NN}\subset H(X\times Y)$ has a convergent subsequence which we denote in the same way to avoid problems with indices. Let $R$ be the limit of such sequence. It means that
    \[
    \lim_{n\to \infty}d_{\H}^{\dd}(R_n,R)=0.
    \]

    Then, by Lemma \ref{lema2.1},
    \begin{equation}\label{eq1}
    \left|\dis(R_n)-\dis(R)\right|\leq4 d_{\H}^{\dd}(R_n,R)
    \end{equation}

    Using the triangle inequality we get
    \[
    \left| d_{\GH}((X_n,A_n),(Y_n,B_n))-d_{\GH}((X,A),(Y,B))\right|\leq d_{\GH}((X_n,A_n),(X,A))+d_{\GH}((Y_n,B_n),(Y,B))
    \]
    and then
    \begin{equation}\label{eq2}
        \dis(R_n)=2d_{\GH}((X_n,A_n),(Y_n,B_n))\to 2d_{\GH}((X,A),(Y,B))
    \end{equation}
    as $n\to\infty$. Using (\ref{eq1}) and (\ref{eq2}) we obtain $\dis(R)= 2 d_{\GH}((X,A),(Y,B))$.

    Finally we have to prove that $R$ is actually a pair correspondence. For every $n\in\NN$, using \cite[Lemma 2.1]{chowmemo} we get
    \begin{alignat*}{2}
    &d_{\H}^{d_X}(X,\pi_1(R))&&\leq d_{\H}^{d_X}(X,\pi_1(R_n))+d_{\H}^{d_X}(\pi_1(R_n),\pi_1(R))\leq d_{\H}^{d_X}(X,X_n)+d_{\H}^{\dd}(R_n,R),\\
    &d_{\H}^{d_X}(A,\pi_1(R|_{A\times B}))&&\leq d_{\H}^{d_X}(A,\pi_1(R_n|_{A_n\times B_n}))+d_{\H}^{d_X}(\pi_1(R_n|_{A_n\times B_n}),\pi_1(R|_{A\times B}))\\&\mbox{}&&\leq d_{\H}^{d_X}(A,A_n)+d_{\H}^{\dd}(R_n|_{A_n\times B_n},R|_{A\times B}),\\
    &d_{\H}^{d_Y}(Y,\pi_2(R))&&\leq d_{\H}^{d_Y}(Y,\pi_2(R_n))+d_{\H}^{d_Y}(\pi_2(R_n),\pi_2(R))\leq d_{\H}^{d_Y}(Y,Y_n)+d_{\H}^{\dd}(R_n,R)
    \end{alignat*}
    
    and
    \begin{align*}
    d_{\H}^{d_Y}(B,\pi_2(R|_{A\times B}))&\leq d_{\H}^{d_Y}(B,\pi_2(R_n|_{A_n\times B_n}))+d_{\H}^{d_Y}(\pi_2(R_n|_{A_n\times B_n}),\pi_2(R|_{A\times B}))\\&\leq d_{\H}^{d_Y}(B,B_n)+d_{\H}^{\dd}(R_n|_{A_n\times B_n},R|_{A\times B}),
    \end{align*}
    where $\pi_1$ and $\pi_2$ are the natural projections of $X\times Y$ on $X$ and $Y$, respectively. Since the right terms can be arbitrary small, all the left terms are zero. Thus, $X=\bar{\pi}_1(R)$, $A=\bar{\pi}_1(R|_{A\times B})$, $Y=\bar{\pi}_2(R)$ and $B=\bar{\pi}_2(R|_{A\times B})$. Since $R$ is compact, the sets $\pi_1(R)$, $\pi_1(R|_{A\times B})$, $\pi_2(R)$ and $\pi_2(R|_{A\times B})$ are closed. Therefore, $\pi_1(R)=X$, $\pi_2(R)=Y$, $\pi_1(R|_{A\times Y})=A$ and $\pi_2(R|_{A\times B})=B$, which means that $R$ is a pair correspondence.
    \end{proof}

    \begin{theorem}{\label{thm:geodesic}}
        The metric space $(\GH_1,d_{\GH})$ is geodesic. Moreover, let $(X,A),(Y,B)\in \GH_1$, $R\in \RRR^{\opt}((X,A),(Y,B))$, and a curve $\gamma_R:[0,1]\to \GH_1$ between $(X,A)$ and $(Y,B)$ such that
        \[
        \gamma_R(0):= (X,A),\quad \gamma_R(1):=(Y,B)\quad\text{and}\quad \gamma_R(t):=(R,R|_{A\times B}, d_{\gamma_R(t)})\quad\text{for}\quad t\in(0,1),
        \]
        where for each $(x,y),(x',y')\in R$ and $t\in (0,1)$, the metric $d_{\gamma_R(t)}$ is given by
        \[
        d_{\gamma_R(t)}((x,y),(x',y')):=(1-t)\,d_X(x,x')+t\,d_Y(y,y').
        \]Then, $\gamma_R$ is a geodesic.
    \end{theorem}
    \begin{proof}
    We have to prove that $\gamma_R$ is a geodesic, i.e., for every $s,t\in[0,1]$ we have to check that
    \[
    d_{\GH}(\gamma_R(s),\gamma_R(t))=|s-t|\,d_{\GH}((X,A),(Y,B)).
    \]
    Using \cite[Lemma 1.3]{chowmemo}, it is sufficient to prove
    \[
    d_{\GH}(\gamma_R(s),\gamma_R(t))\leq|s-t|\,d_{\GH}((X,A),(Y,B)).
    \]
    We have three cases depending on whether $s$ or $t$ are the end points. If $s,t\in (0,1)$, let $\tri\in \RRR(\gamma_R(t),\gamma_R(s))$ be the diagonal pair correspondence. Then,
        \begin{align*}
            \dis(\tri)&=\frac{1}{2}\left( \sup_{(a,a),(b,b)\in\tri}\left| d_{\gamma_R(t)}(a,b)-d_{\gamma_R(s)}(a,b) \right|\right.\\
            & \hspace{5.5cm}\left.+\sup_{(c,c),(d,d)\in\tri|_{A\times B\times A\times B}}\left| d_{\gamma_R(t)}(c,d)-d_{\gamma_R(s)}(c,d) \right|\right)\\
            &=\frac{1}{2}\left( \sup_{(x,y),(x',y')\in R}\left| d_{\gamma_R(t)}((x,y),(x',y'))-d_{\gamma_R(s)}((x,y),(x',y')) \right|\right.\\
            & \hspace{2.5cm}\left.+\sup_{(z,w),(z',w')\in R|_{A\times B}}\left| d_{\gamma_R(t)}((z,w),(z',w'))-d_{\gamma_R(s)}((z,w),(z',w')) \right|\right)
            \end{align*}
            \begin{align*}
            &=\frac{1}{2}\left( \sup_{(x,y),(x',y')\in R}\left| (1-t)d_X(x,x')+td_Y(y,y')-(1-s)d_X(x,x')-sd_Y(y,y') \right|\right.\\
            & \left.+\sup_{(z,w),(z',w')\in R|_{A\times B}}\left| (1-t)d_X(z,z')+t d_Y(w,w')-(1-s)d_X(z,z')-sd_Y(w,w') \right|\right)\\
            &=\frac{1}{2}\left( \sup_{(x,y),(x',y')\in R}\left| (s-t)d_X(x,x')+(s-t)d_Y(y,y') \right|\right.\\
            & \hspace{4cm}\left.+\sup_{(z,w),(z',w')\in R|_{A\times B}}\left| (s-t)d_X(z,z')+(s-t) d_Y(w,w') \right|\right)\\
            &=\left|t-s\right|\,\frac{1}{2} \left( \sup_{(x,y),(x',y')\in R}\left| d_X(x,x')+d_Y(y,y') \right|\right.\\
            & \hspace{6cm}\left.+\sup_{(z,w),(z',w')\in R|_{A\times B}}\left| d_X(z,z')+ d_Y(w,w') \right|\right)\\
            &=\left|t-s\right|\,\dis(R)=\left|t-s\right|\,2\,d_{\GH}((X,A),(Y,B)).
            \end{align*}
            By theorem \ref{teo1}, $d_{\GH}(\gamma_R(t),\gamma_R(s))\leq\frac{1}{2}\dis(\tri)=|t-s|\,d_{\GH}((X,A),(Y,B))$.

    If $s=0$ and $t\in(0,1)$, we define $R_X:=\left\{ (x,(x,y)):(x,y)\in R\right\}$, which is a pair correspondence in $\RRR((X,A),\gamma_R(t))$. Therefore,
    \begin{align*}
            \dis(R_X)&=\frac{1}{2}\left( \sup_{(x,(x,y),(x'(x',y'))\in R_X}\left| d_X(x,x')-d_{\gamma_R(t)}((x,y),(x',y')) \right|\right.\\
            & \hspace{3.5cm}\left.+\sup_{(a,(a,b)),(a',(a',b'))\in R_X|_{A\times  A\times B}}\left| d_X(a,a')-d_{\gamma_R(t)}((a,b),(a',b')) \right|\right)\\
           &=\frac{1}{2}\left( \sup_{(x,(x,y),(x'(x',y'))\in R_X}\left| d_X(x,x')-(1-t)d_X(x,x')-t\,d_Y(y,y') \right|\right.\\
            & \hspace{2.3cm}\left.+\sup_{(a,(a,b)),(a',(a',b'))\in R_X|_{A\times  A\times B}}\left| d_X(a,a')-(1-t)d_X(a,a')-t\,d_Y(b,b') \right|\right)\\
            &=\frac{1}{2}\left( \sup_{(x,(x,y),(x'(x',y'))\in R_X}\left| t\,d_X(x,x')-t\,d_Y(y,y') \right|\right.\\
            & \hspace{4.9cm}\left.+\sup_{(a,(a,b)),(a',(a',b'))\in R_X|_{A\times  A\times B}}\left| t\,d_X(a,a')-t\,d_Y(b,b') \right|\right)\\
            &=\frac{t}{2}\left( \sup_{(x,(x,y),(x'(x',y'))\in R_X}\left| d_X(x,x')-d_Y(y,y') \right|\right.\\
            & \hspace{5.2cm}\left.+\sup_{(a,(a,b)),(a',(a',b'))\in R_X|_{A\times  A\times B}}\left| d_X(a,a')-d_Y(b,b') \right|\right)\\
            &=t\,\dis(R)=2t\,d_{\GH}((X,A),(Y,B)).
            \end{align*}
            Again, by theorem \ref{teo1}, we get $d_{\GH}((X,A),\gamma_R(t))\leq t\, d_{\GH}((X,A),(Y,B))$.

    The remaining case when $s\in (0,1)$ and $t=1$ is analogous and we obtain $d_{\GH}(\gamma_R(s),(Y,B))\leq |1-s|\, d_{\GH}((X,A),(Y,B)).$
    \end{proof}
\begin{rema}
    Similarly as in theorem \ref{thm:geodesic}, we can construct geodesics in $(\GH_k,d_{\GH})$, defining optimal tuple correspondences and obtaining analogous results to Lemma \ref{lema2.1} and proposition \ref{prop_exi_opt_corr}.
\end{rema}

\subsection{Separability}
\mbox{}%
\vspace{5pt}

\begin{theorem}{\label{thm:separable}}
    The metric space $(\GH_1,d_{\GH})$ is separable.
\end{theorem}

\begin{proof}
Let $\calX_n$ be the set of finite metric spaces with $n$ elements and rational distances between their points, and set $\calX = \bigcup_{n\in \NN} \calX_n$. Clearly each $\calX_n$ is countable, therefore $\calX$ is countable. Moreover, $\calX$ is dense in $(\GH,d_{\GH})$. Now, if $(X,A)\in \GH_1$ and $\eps>0$, there is some $X'\in \calX$ such that $d_{\GH}(X,X')<\varepsilon$, which means there exists a correspondence $R\subset X\times X'$ such that
\[
\overline{\dis}(R) < 2\eps.
\] 
Now, let 
\[
A' = \{a'\in X':(a,a')\in R,\ \text{ for some }\ a\in A\}.
\]
We claim that $d_{\GH}((X,A),(X',A'))< \eps$. Indeed, by construction, $R$ is a pair correspondence between $(X,A)$ and $(X',A')$, and
\begin{align*}
\dis(R) &= \frac{1}{2}\left(\sup\left\{ \left|d_{X}(x,x')-d_{Y}(y,y')\right|\colon (x,y),(x',y')\in R\right\} \right.\\
&\phantom{asdfg}+\left.\sup\left\{ \left|d_{X}(a,a')-d_{Y}(b,b')\right|\colon (a,b),(a',b')\in R|_{A\times A'}\right\}\right)\\
&\leq \frac{1}{2}(2\overline{\dis}(R) )< 2\eps.
\end{align*}
The claim follows, and this implies that the set 
\[
\calX_{\GH_1}= \{(X,A)\in \GH_1: X\in \calX,\ A\subset X\}
\]
is dense in $\GH_1$, and it is countable since $\calX$ is countable and each $X\in \calX$ has finitely many subsets.
\end{proof}
\begin{rema}
    The space of compact metric $k$-tuples $(\GH_k,d_{\GH})$ is also separable, for any $k\in \NN$, by an analogous argument.
\end{rema}

\subsection{Arzelà--Ascoli theorem for $\GH_1$}
\mbox{}%

\vspace{5pt}
In this section we prove a version of the classical Arzel\`a--Ascoli theorem for metric pairs, generalizing the one presented in \cite{herron2016} for maps between pointed metric spaces.
\begin{definition}\label{def:relative map}
Let $(X,A)$ and $(Y,B)$ be metric pairs. A \textit{relative map} $f\colon (X,A) \to (Y,B)$ is a map $f\colon X\to Y$ such that $f(A)\subset B$. The \textit{Lipschitz constant} of $f$ is given by
\[
\Lip(f)=\inf\left\{K\in \RR\cup \{\infty\}\left|\frac{d_Y(f(x_1),f(x_2))}{d_X(x_1,x_2)}\leq K \text{ for every }x_1\neq x_2\right.\right\},
\]
and if $\Lip(f)<\infty$, then $f$ is a \textit{Lipschitz map}. 
\end{definition}

\begin{definition}\label{def:locally uniformly bounded}
A family of relative maps $\{f_i\colon (X_i,A_i)\to (Y_i,B_i)\}_{i\in I}$ is \textit{relatively uniformly bounded} if for any $R>0$ there exists $T>0$ such that for all $i\in I$, $f_i(\BB_R(A_i))\subset \BB_T(B_i)$.
\end{definition}

\begin{definition}\label{def: equicontinuous family of relative maps}
A family of relative maps $\{f_i\colon (X_i,A_i)\to (Y_i,B_i)\}_{i\in I}$ is \textit{relatively equicontinuous} if for any $R>0$ and $\eps>0$ there exist $\delta>0$ such that for all $i\in I$, and all $x,y\in \BB_R(A_i)$, $d_{X_i}(x,y)< \delta$ implies $d_{Y_i}(f_i(x),f_i(y))< \eps$.
\end{definition}


\begin{definition}\label{def:relative uniform convergence}
A sequence of relative maps $\{f_i\colon (X_i,A_i)\to (Y_i,B_i)\}_{i\in \NN}$ \textit{converges relatively uniformly} to a relative map $f_\infty\colon (X_\infty,A_\infty)\to (Y_\infty,B_\infty)$ if there exist pairs $(\calX,\calA)$ and $(\calY,\calB)$, and isometric embeddings $\varphi_i\colon (X_i,A_i)\to (\calX,\calA)$ and $\psi_i\colon (Y_i,B_i)\to (\calY,\calB)$, for $i\in \NN\cup \{\infty\}$, such that, for each $R>0$ and all $\eps>0$, there exists $N\in \NN$ and $\eta>0$ such that, for all $i\geq N$, all $x\in \BB_R(A_i)$, $x_\infty\in \BB_R(A_\infty)$, if $d(\varphi_i(x),\varphi_\infty(x_\infty)) < \eta$ then $d(\psi_i(f_i(x)),\psi_\infty(f_\infty(x_\infty))) < \eps$.
\end{definition}

\begin{theorem}{\label{thm:arzela}}
Consider metric pairs $(X_i,A_i)\toGH (X_\infty,A_\infty)$ and $(Y_i,B_i)\toGH (Y_\infty,B_\infty)$, where $A_i$ is compact for all $i\in \NN\cup \{\infty\}$, and relatively equicontinuous, relatively uniformly bounded maps $f_i\colon (X_i,A_i)\to (Y_i,B_i)$, for all $i\in \NN$. Then there exists a subsequence of $\{f_i\}_{i\in\NN}$ converging relatively uniformly to a continuous relative map $f_\infty\colon (X_\infty,A_\infty)\to (Y_\infty,B_\infty)$. Moreover, if $\Lip(f_i)\leq K$ for all $i\in \NN$ then $\Lip(f_\infty)\leq K$, and if $\Lip(f_i|_{A_i})\leq L$ for all $i\in \NN$ then $\Lip(f_\infty|_{A_\infty})\leq L$. 
\end{theorem}

\begin{proof}
By \cite[Theorem 1.1]{aagmc} we know that there are metric pairs $(\calX,\calA)$, $(\calY,\calB)$ and isometric embeddings $\varphi_i \colon (X_i,A_i)\to (\calX,\calA)$ and $\psi_i \colon (Y_i,B_i)\to (\calY,\calB)$ such that 
\[
d_\H^\calX(\varphi_i(X_i),\varphi_\infty(X_\infty))\to 0
\]
and 
\[
d_\H^\calY(\psi_i(Y_i),\psi_\infty(Y_\infty))\to 0.
\]
The rest of the proof follows along the same lines as that of \cite[Proposition 5.1]{herron2016}. The only detail to take into account is that the compactness of $A_i$ and the properness of $X_i$ for $i\in \NN\cup \{\infty\}$ imply that the neighborhoods $\BB_R(A_i)$ ar compact for any $R>0$.

Finally, we obtain the last assertions from the proof of \cite[Proposition 5.1]{herron2016}. Namely, if $f_i$ (or their restrictions to $A_i$) are $L$-Lipschitz, let $u, v\in X_\infty$ (respectively, $u,v\in A_\infty$) and consider sequences $\{u_i\}_{i\in \NN}$ and $\{v_i\}_{i\in \NN}$ such that $u_i,v_i\in X_i$ (or $u_i,v_i\in A_i$, respectively), and $\varphi_i(u_i)\to u$, $\psi_i(v_i)\to v$. By the triangle inequality,
\begin{align*}
d(f_\infty(u),f_\infty(v)) &\leq d(\psi_\infty(f_\infty(u)),\psi_i(f_i(u_i)))+ d(f_i(u_i),f_i(v_i))+d(\psi_i(f_i(v_i)),\psi_\infty(f_\infty(v)))\\
&\leq  d(\psi_\infty(f_\infty(u)),\psi_i(f_i(u_i)))+ Ld(u_i,v_i)+d(\psi_i(f_i(v_i)),\psi_\infty(f_\infty(v))).
\end{align*}
By taking a subsequence of $\{f_i\}_{i\in \NN}$ pointwise convergent to $f_\infty$, we can let $i\to \infty$ and obtain that 
\[
d(f_\infty(u),f_\infty(v)) \leq L\ d(u,v).\qedhere
\]
\end{proof}
\begin{rema}
    If we take analogous tuple versions of definitions \ref{def:relative map}, \ref{def:locally uniformly bounded}, \ref{def: equicontinuous family of relative maps} and \ref{def:relative uniform convergence}, we get a tuple version of theorem \ref{thm:arzela}.
\end{rema}

\section{Approximation by surfaces}\label{s:cassorla}

In \cite{cassorla}, Cassorla proved that any compact length space can be approximated by Riemannian surfaces. In this section, we prove a generalization of that result but for metric  pairs.

\begin{theorem}\label{thm:cassorla}
    If $\calM$ is the set of compact length metric pairs and if $\calS$ is the subset of $\calM$ consisting of 2-dimensional Riemannian manifolds which can be isometrically embedded in $\RR^3$ together with a 2-dimensional submanifold with boundary, then $\bar{\calS}=\calM$ in $d_{\GH_1}$.
\end{theorem}
\begin{proof}The proof is separated in several steps as follows.  First, for any compact metric pair $(X,A)$ we construct a 1-dimensional simplicial complex $L$ together with a subcomplex $K$ such that it is close to it. After that, we build a surface $S$ with a 2-dimensional submanifold with boundary $P$ which are close to $(L,K)$. Finally, using triangle inequality, we can make $(S,P)$ as close as we want to $(X,A)$ proving the density.
\subsubsection*{1-complex construction}

We begin by taking a compact length metric pair $(X,A)$. For every integer $n>3$ we choose a minimal $10^{-n}$-dense set \[A_n=\{x_i=x_i(n)\}_{i=1}^{a(n)}\subset A.\] Recall that $D\subset X$ is a $\nu$-\textit{dense set} if for each $x \in X$ there exists $a(x) \in D$ so that $d_X(x,a(x)) < \nu$. The existence of a minimal $\nu$-dense set, for any $\nu>0$ follows from the compactness of $X$ (see, for example, \cite[Exercise~1.6.4]{buragoburagoivanov}).

Furthermore, $D$ is a \textit{minimal} $\nu$-dense set if $D$ is $\nu$-dense and if $a_1,a_2 \in D$, $a_1\neq a_2$, then $d_X(a_1,a_2)>\nu$. We define the 1-complex $K=K(X,A,n)$ with vertices labeled by $A_n$ and 1-simplexes connecting $x_i$ to $x_j$ of length $d_X(x_i,x_j)$ if and only if 
\[
d_X(x_i,x_j)<\frac{2^n}{10^n}=\frac{1}{5^n}
\]
and there exists a geodesic segment within $A$.

Now, we extend the set $A_n$ to a minimal $10^{-n}$-dense set  $X_n=\{ x_i\}_{i=1}^{a(n)}\bigcup\{ x_i\}_{i=a(n)}^{h(n)}$ of $X$. We also define the extension of the 1-complex $L=L(X,A,n)$ with vertices labeled by $X_n$ and 1-simplices connecting $x_i$ to $x_j$ of length $d_X(x_i,x_j)$ if and only if 
\[
d_X(x_i,x_j)<\frac{2^n}{10^n}=\frac{1}{5^{n}}.
\]

We give $L$ the induced piecewise linear metric denoted by $d_L$. We observe that $K$ is a subcomplex of $L$ and we also use $d_L$. 

The following lemmas help us to compute the distances between the complexes and the spaces.

\begin{lemma}[Lemma 2.3 of \cite{cassorla}]\label{lema2}
    Assuming the same notation as in the proof of theorem \ref{thm:cassorla}. For all $i$ and $j$, 
    \[
        d_L(x_i,x_j)<2^{\mu}d_X(x_i,x_j)+5^{-n},
    \]
 where $\mu=\mu(n)$ is chosen to satisfy $2^{n-\mu}+2<2^{n}$ when $\mu(n)\to 0$ as $n\to\infty$.
\end{lemma}

\begin{lemma}[Example 2, section 2 of \cite{petersen}]\label{lema1} 
    Let $X$ and $Y$ be two compact metric spaces. Suppose that there are $\{v_i\}_{i=1}^{r}\subset X$ and $\{w_i\}_{i=1}^{r}\subset Y$ $\varepsilon$-nets with $\varepsilon>0$ such that 
    \[
    \left| d_X(v_i,v_j)-d_Y(w_i,w_j)\right|<\varepsilon,
    \]
    for all $i,j\in\{1,\dots,r\}$. Then $d_{\GH}(X,Y)\leq 2\varepsilon$.
\end{lemma}

Combining these two lemmas with the observation that $d_L(x_i,x_j)\geq d_X(x_i,x_j)$, for all $i,j$, we obtain an estimate of the distance between the complex pair and the original metric pair.

\begin{equation}\label{coro}
    d_{\GH}((X,A),(L,K))<\left(2^{\mu(n)}-1\right)\diam(X)+5^{-n}.
\end{equation}

\subsubsection*{Construction of the surface}

Now, the second step is to construct a surface near to the simplicial pair $(L,K)$. Let $0<\varepsilon<10^{-(n+1)}$, choose a minimal $\varepsilon$-dense set $\{y_k\}_{k=1}^{v}$ in $K$ and we extend it such that  $ \{y_k\}_{k=1}^{v} \subset \{y_k\}_{k=v+1}^{l}$ is a minimal $\varepsilon$-dense set of $L$ and these sets do not contain any vertices. Since $L$ is 1-dimensional and using Menger's \emph{Einbettungssatz} (Embedding's theorem)(cf. page 285 of \cite{menger}), we embed smoothly in $B_{\tilde{m}}(\bar{0})\subset\RR^3$ with $\tilde{m}=\min_{i\neq j}\{d(x_i,x_j)\}$. Then we extend the image of the simplexes such that the length in $X$ is preserved.

Let us consider $m=\min_{i,k}d_{L}(x_i,y_k)$ and $c_{ij}$ is the image in $\RR^3$ of a 1-simplex in $L$ that joins $x_i$ and $x_j$ if it exists. From now on in this construction, we will consider $L$ to be inside of $\RR^3$ and we measure the distances in $\RR^3$. Additionally, if $D\subset\RR^3$, then $d_D$ is computed as the infimum of lengths of curves in $D$.

If $c_{ij}$ exists, we take $r_{ij}>0$ such that if $r_{ij}>r>0$, we can isolate the vertices, i. e.,
\begin{equation}
    \partial B_r(x_i)\cap c_{ij}=\{p_{ij}(r)\},
\end{equation}
is a point, we can see this configuration illustrated in Figure \ref{fig2}, and
\begin{equation}
    B_r(x_i)\cap c_{kg}=\emptyset
\end{equation}
for $k\neq i$.

\begin{figure}[ht]
\centering
\begin{tikzpicture}
  \coordinate [label=left:$x_i$]  (A) at (0,0);
  \coordinate [label=-30:$p_{ij}(r)$] (B) at (1.25,0.25);
  \coordinate[label=right:$x_j$] (C) at (5,1);
  \coordinate[label=above:$c_{ij}$] (D) at (3.75,0.75);
  
  \draw (A) -- (C);

   \fill (A)  circle (2pt);
   \fill (B) circle (2pt);
   \fill (C)  circle (2pt);

  \draw[loosely dashed] (A) let
              \p1 = ($ (B) - (A) $)
            in
              circle ({veclen(\x1,\y1)});

  \def\radius{1.27}
  \coordinate (E) at ({\radius*cos(100)},{\radius*sin(100)});

  \draw (A) -- (E);
  \node[fill=white] at ($ (A)!0.5!(E) $) {$r$};
\end{tikzpicture}
\caption{}
\label{fig1}
\end{figure}
If $c_{ij}$ does not exists, let $\{p_{ij}(r)\}=\emptyset$. Also, we choose 
\[
0<\bar{r}<\min\left\{\frac{r_{ij}}{3},\frac{m}{3},\frac{\varepsilon}{8\pi l}\right\}
\]
so we isolate $x_i$ from the rest of $x_j$'s and $y_k$'s, and the size of the ball is sufficiently small. 

Let $p_{ij}=p_{ij}(\bar{r})$. For each $i$ and $t>0$, let $T_i(t)=\partial B_{\bar{r}}(x_i)\smallsetminus\cup_j B_{t}(p_{ij})$, illustrated in Figure \ref{fig2}.  We see that $T_i(t)$ approaches $\partial B_{\bar{r}}(x_i)\smallsetminus \cup_j \{p_{ij}\}$ as $t\to 0$.

\begin{figure}[ht]
\centering
\begin{tikzpicture}

    \coordinate [label=left:$x_i$]  (A) at (0,0);
    \fill (A)  circle (2pt);
    
    \def\radius{2}
    \draw[line width=1pt] (2,0) arc (0:40:2);
    \draw[dashed] ({\radius*cos(40)},{\radius*sin(40)}) arc (40:80:2);
    \draw[line width=1pt] ({\radius*cos(80)},{\radius*sin(80)}) arc (80:120:2);
    \draw[dashed] ({\radius*cos(120)},{\radius*sin(120)}) arc (120:160:2);
    \draw[line width=1pt] ({\radius*cos(160)},{\radius*sin(160)}) arc (160:200:2);
    \draw[dashed] ({\radius*cos(200)},{\radius*sin(200)}) arc (200:240:2);
    \draw[line width=1pt] ({\radius*cos(240)},{\radius*sin(240)}) arc (240:280:2);
    \draw[dashed] ({\radius*cos(280)},{\radius*sin(280)}) arc (280:320:2);
    \draw[line width=1pt] ({\radius*cos(320)},{\radius*sin(320)}) arc (320:360:2);

    \coordinate (pij) at ({\radius*cos(60)},{\radius*sin(60)});
    \fill (pij)  circle (2pt);
    \draw[dashed] (pij) circle (0.7);
    \node at ($(pij) + (1.7,0)$) {$B_t(p_{ij})$};
    
    \coordinate  (pik) at ({\radius*cos(140)},{\radius*sin(140)});
    \fill (pik)  circle (2pt);
    \draw[dashed] (pik) circle (0.7);
    
    \coordinate  (pih) at ({\radius*cos(220)},{\radius*sin(220)});
    \fill (pih)  circle (2pt);
    \draw[dashed] (pih) circle (0.7);

    \coordinate  (pil) at ({\radius*cos(300)},{\radius*sin(300)});
    \fill (pil)  circle (2pt);
    \draw[dashed] (pil) circle (0.7);
    \node at ($(pil) + (1.7,0)$) {$B_t(p_{il})$};

    \node at (2.5,0) {$T_i(t)$};
    
\end{tikzpicture}
\caption{}
\label{fig2}
\end{figure}
For each $i$ we take $t_i$ such that, if $t_i>t>0$, then 
\begin{equation}
    B_t(p_{ij})\cap B_t(p_{ik})=\emptyset,
\end{equation}
$j\neq k$, and for any $v,w\in T_i(t)$,
\begin{equation}\label{eq_ti}
    d_{T_i(t)}(v,w)\leq d_{\partial B_{x_i}(\bar{r})}(v,w)+\pi \bar{r}.
\end{equation}
That is, we isolate the holes in $T_i(t)$ and we make them sufficiently small. We also take $t_0$ such that if $t\in (0,t_0)$, we have the following properties:
\begin{enumerate}[label=(\alph*)]
    \item $C_{ij}(t)=\left\{z\in\RR^3| d(z,c_{ij})=t\right\}$ is a submanifold outside of $B_{\bar{r}}(x_i)$ and $B_{\bar{r}}(x_j)$ for every $i$ and $j$, this setting can be seen in Figure \ref{fig3}.
    \begin{figure}[ht]
    \centering
    \begin{tikzpicture}
    \coordinate [label=below:$x_i$]  (xi) at (-3,-0.5);
    \coordinate [label=above:$x_j$]  (xj) at (3,0.5);
    \draw (xi) -- (xj);
    \fill (xi) circle (2pt);
    \fill (xj) circle (2pt);

    \draw[dashed] (xi) ellipse (0.4 and 1);
    \draw ($(xi)+({cos(90)},{sin(90)})$) arc (90:270:0.8 and 1);
    \draw[dashed] (xj) ellipse (0.4 and 1);
    \draw ($(xj)+({cos(270)},{sin(270)})$) arc (270:450:0.8 and 1);

    \coordinate (ai) at ($(xi)+(0,1)$);
    \coordinate (bi) at ($(xi)-(0,1)$);
    \coordinate (aj) at ($(xj)+(0,1)$);
    \coordinate (bj) at ($(xj)-(0,1)$);
    \draw (ai)--(aj);
    \draw (bi)--(bj);

    \node at ($(xi)+(0,1.35)$) {$C_{ij}(t)$};
\end{tikzpicture}
\caption{}
\label{fig3}
\end{figure}
    \item The tubes with a common endpoint only intersect inside the balls around the corresponding vertex,
    \begin{equation}\label{eq_b}
    C_{ij}(t)\cap C_{ig}(t)\subset B_{\bar{r}}(x_i),
    \end{equation}
    for $j\neq g$.
    \item The tubes without common endpoints do not intersect,
    \begin{equation}\label{eq_c}
        C_{ij}(t)\cap C_{kg}(t)=\emptyset,
    \end{equation}
    for $i\neq k$ and $j\neq g$.

    \item We define $S_k(t)=\left\{z\in C_{ij}|d(y_k,z)=t\right\}$, which are shown in Figure \ref{fig4}.
    \begin{figure}[ht]
    \centering
\begin{tikzpicture}
    \coordinate [label=below:$x_i$]  (xi) at (-3,-0.5);
    \coordinate [label=above:$x_j$]  (xj) at (3,0.5);
    \coordinate [label=below:$y_k$] (yk) at (1,0.1666667);
    \draw (xi) -- (xj);
    \fill (xi) circle (2pt);
    \fill (xj) circle (2pt);
    \fill (yk) circle (2pt);

    \draw[dashed] (xi) ellipse (0.4 and 1);
    \draw ($(xi)+({cos(90)},{sin(90)})$) arc (90:270:0.8 and 1);
    \draw[dashed] (xj) ellipse (0.4 and 1);
   \draw ($(xj)+({cos(270)},{sin(270)})$) arc (270:450:0.8 and 1);
    \draw[line width=1pt] (yk) ellipse (0.4 and 1);

    \node at (1,-1.35) {$S_k(t)$};

    \coordinate (ai) at ($(xi)+(0,1)$);
    \coordinate (bi) at ($(xi)-(0,1)$);
    \coordinate (aj) at ($(xj)+(0,1)$);
    \coordinate (bj) at ($(xj)-(0,1)$);
    \draw (ai)--(aj);
    \draw (bi)--(bj);

    \node at ($(xi)+(0,1.35)$) {$C_{ij}(t)$};
\end{tikzpicture}
\caption{}
\label{fig4}
\end{figure} 
    If $y_k$ and $y_g$ are on $c_{ij}$, then the comparison of the distance between these sets and the distance between the centers is small,
    \begin{equation}\label{eq_d}
        \left|d_{C_{ij}(t)}(S_k(t),S_g(t))-d_L(y_k,y_g)\right|<\frac{\varepsilon}{2l}
    \end{equation}   
    for every $i$, $j$, $k$ and $g$. 
    \item The comparison between the distance the any $S_k(t)$ and the ball around the end points, and the distance between the center of the set and the center of the corresponding hole is also small,
    \begin{equation}\label{eq_e}
        \left|d_{C_{ij}(t)}(S_k(t),B_{\bar{r}}(x_i))-d_L(y_k,p_{ij})\right|<\frac{\varepsilon}{4l}
    \end{equation}
    for every $i$, $j$ and $k$.
\end{enumerate}
By (\ref{eq_b}), $t_0<\bar{r}$ and $C_{ij}(t)\cap B_{\bar{r}}(x_k)=\emptyset$ for $0<t<t_0$ and $i\neq k\neq j$. Also by (\ref{eq_b}), the sets $S_k(t)$ are actually circles of radius $t$ centered at $y_k$. Thanks to this and since $c_{ij}$ is smoothly embedded, (\ref{eq_d}) and (\ref{eq_e}) can always be satisfied. Thus, we set 
\[
0<\bar{t}<\min\left\{ t_i,\frac{\varepsilon}{8\pi l} \right\}.
\]

Now we begin the construction of a surface that is $10^{-n}$-close to $L$ by defining a metric space $S$ as follows. We replace each vertex $x_i$ in $L$ by $T_i(\bar{t})$ and we replace each $c_{ij}$ by $C_{ij}(\bar{t})$, up to the circles of intersection with $T_{i}(\bar{t})$ and $T_{j}(\bar{t})$. We illustrate this construction in Figure \ref{fig5}. Thanks to (\ref{eq_b}) and (\ref{eq_c}), $S$ does not have self-intersections. Also we notice that $S$ is not smooth in $N$, a finite number of circles, because of the way we construct it. Finally, we use the intrinsic metric on $S$.

\begin{figure}[ht]
\centering
\begin{tikzpicture}

    \coordinate [label=left:$x_i$]  (A) at (0,0);
    \fill (A)  circle (2pt);
    
    \def\radius{2}
    \draw[line width=1pt] (2,0) arc (0:40:2);
    \draw[line width=1pt] ({\radius*cos(80)},{\radius*sin(80)}) arc (80:120:2);
    \draw[line width=1pt] ({\radius*cos(160)},{\radius*sin(160)}) arc (160:200:2);
    \draw[line width=1pt] ({\radius*cos(240)},{\radius*sin(240)}) arc (240:280:2);
    \draw[line width=1pt] ({\radius*cos(320)},{\radius*sin(320)}) arc (320:360:2);

    \draw[dashed,rotate=90] (A) ellipse (0.8 and \radius);

    \coordinate (pij) at ({\radius*cos(60)},{\radius*sin(60)});
    \fill (pij) circle (2pt);
    \coordinate (pij1) at ({\radius*cos(40)},{\radius*sin(40)});
    \coordinate (pij2) at ({\radius*cos(80)},{\radius*sin(80)});

    \coordinate  (pik) at ({\radius*cos(140)},{\radius*sin(140)});
    \fill (pik) circle (2pt);
    \coordinate (pik1) at ({\radius*cos(120)},{\radius*sin(120)});
    \coordinate (pik2) at ({\radius*cos(160)},{\radius*sin(160)});
    
    \coordinate  (pih) at ({\radius*cos(220)},{\radius*sin(220)});
    \fill (pih) circle (2pt);
    \coordinate (pih1) at ({\radius*cos(200)},{\radius*sin(200)});
    \coordinate (pih2) at ({\radius*cos(240)},{\radius*sin(240)});

    \coordinate  (pil) at ({\radius*cos(300)},{\radius*sin(300)});
    \fill (pil) circle (2pt);
    \coordinate (pil1) at ({\radius*cos(280)},{\radius*sin(280)});
    \coordinate (pil2) at ({\radius*cos(320)},{\radius*sin(320)});

    \node at (2.5,0) {$T_i(t)$};

    \coordinate (a1) at ({3.5*cos(60)},{3.5*sin(60)});
    \coordinate (pij11) at ({3.5*cos(48.5)},{3.5*sin(48.5)});
    \coordinate (pij22) at ({3.5*cos(71.5)},{3.5*sin(71.5)});
    \draw (pij1)--(pij11);
    \draw (pij2)--(pij22);
    
    \draw[dashed,rotate=60] (a1) ellipse (0.3 and 0.7);
    \draw[dashed,rotate=60] (pij) ellipse (0.3 and 0.7);

    \coordinate (a2) at ({3.5*cos(140)},{3.5*sin(140)}); 
    \coordinate (pik11) at ({3.5*cos(128.5)},{3.5*sin(128.5)});
    \coordinate (pik22) at ({3.5*cos(151.5)},{3.5*sin(151.5)});
    \draw (pik1)--(pik11);
    \draw (pik2)--(pik22);
    
    \draw[dashed,rotate=140] (a2) ellipse (0.3 and 0.7);
    \draw[dashed,rotate=140] (pik) ellipse (0.3 and 0.7);

    \coordinate (a3) at ({3.5*cos(220)},{3.5*sin(220)});
    \coordinate (pih11) at ({3.5*cos(208.5)},{3.5*sin(208.5)});
    \coordinate (pih22) at ({3.5*cos(231.5)},{3.5*sin(231.5)});
    \draw (pih1)--(pih11);
    \draw (pih2)--(pih22);
    \draw[dashed,rotate=220] (a3) ellipse (0.3 and 0.7);
    \draw[dashed,rotate=220] (pih) ellipse (0.3 and 0.7);

    \coordinate (a4) at ({3.5*cos(300)},{3.5*sin(300)});
    \coordinate (pil11) at ({3.5*cos(288.5)},{3.5*sin(288.5)});
    \coordinate (pil22) at ({3.5*cos(311.5)},{3.5*sin(311.5)});
    \draw (pil1)--(pil11);
    \draw (pil2)--(pil22);
    \draw[dashed,rotate=300] (a4) ellipse (0.3 and 0.7);
    \draw[dashed,rotate=300] (pil) ellipse (0.3 and 0.7);

    \node at (0.35,3) {$C_{ij}(t)$};

    \coordinate (b1) at ({3.5*cos(60)},{3.5*sin(60)});
    \coordinate (b2) at ({3.5*cos(140)},{3.5*sin(140)}); 
    \coordinate (b3) at ({3.5*cos(220)},{3.5*sin(220)});
    \coordinate (b4) at ({3.5*cos(300)},{3.5*sin(300)});

    \coordinate (c1) at ({4.5*cos(60)},{4.5*sin(60)});
    \coordinate (c2) at ({4.5*cos(140)},{4.5*sin(140)}); 
    \coordinate (c3) at ({4.5*cos(220)},{4.5*sin(220)});
    \coordinate (c4) at ({4.5*cos(300)},{4.5*sin(300)});

    \draw[line width=1pt] (b1) -- (c1);
    \draw[line width=1pt] (b2) -- (c2);
    \draw[line width=1pt] (b3) -- (c3);
    \draw[line width=1pt] (b4) -- (c4);

    \node at ($(c1)+(0.55,0)$) {$c_{ij}(t)$};
    \node at ($(pij)-(0.5,0.6)$) {$p_{ij}(r)$};
    
\end{tikzpicture}
\caption{}
\label{fig5}
\end{figure}

This metric space is the candidate to be the Riemannian surface close to the metric pair. In order to estimate the distance between $S$ and $X$ and apply Lemma \ref{lema1} we choose for each $y_j$ in $L$ a point $z_j\in S_j(\bar{t})$. By definition of $\bar{t}$ and $\bar{r}$,
\begin{equation}\label{dist_sing}
    d(N,\left\{z_j\right\}_{j=1}^{l})>0.
\end{equation}

We notice that any selection of $z_j$'s is a $3\varepsilon$-net in $S$. This is because for any $x\in S$, the length of any shortest curve joining $x$ with $z_f$, the closest of the $z_j$'s, can be bounded as follows:
\[
d(x,z_f)\leq \underbrace{2\varepsilon+\frac{\varepsilon}{2l}+2\pi\bar{t}}_{\text{A}}+\underbrace{2\pi \bar{r}}_{\text{B}}.
\]

Part A bounds the section of the curve in the tube using equation (\ref{eq_d}), the distance in the circle $S_f$ and the minimality of $\{y_j\}_{i=1}^{l}$. And part B is a bound for the section on the sphere obtained with equation (\ref{eq_ti}). Finally, A$+$B is smaller than $3\varepsilon$ due to the definition of $\bar{t}$ and $\bar{r}$. 

Finally, we have to estimate
\[
\delta=\left|d_L(y_i,y_j)-d_S(z_i,z_j)\right|.
\]
We assume first that $y_i$ and $y_j$ can be joined without going through any other $y_k$. Then $d_L(y_i,y_k)<4\varepsilon$. If both points lie on the same 1-simplex, then using (\ref{eq_d}) and the fact the $S_k(\bar{t})$ is a circle of radius $\bar{t}$:
        \[
        \delta\leq \frac{\varepsilon}{2l}+2\pi\bar{t}<\frac{\varepsilon}{2l}+\frac{\varepsilon}{4l}<\frac{\varepsilon}{l}.
        \]
Moreover, if both points do not lie in the same segment, then using (\ref{eq_ti}) and (\ref{eq_e}) we obtain on one hand
        \begin{equation}\label{desiq}
        \left|d_S(S_i(\bar{t}),S_j(\bar{t}))-d_S(z_i,z_j)\right|\leq2\pi \bar{t},
        \end{equation}
because those numbers are almost the same up to the distance between two points in a circle. On the other hand
        \begin{eqnarray}
        \left|d_S(S_i(\bar{t}),S_j(\bar{t}))-d_L(y_i,y_j)\right|&=&\left|d_S(S_i(\bar{t}),S_j(\bar{t}))-d_L(y_i,x_k)-d_L(x_k,y_j)\right|\nonumber\\&\leq&\left|d_S(S_i(\bar{t}),B_{\bar{r}}(x_k))+2\pi \bar{r} +d_S(B_{\bar{r}}(x_k),S_j(\bar{t}))\right.\nonumber\\
        & &\left.-d_L(y_i,x_k)-d_L(x_k,y_j)\right|\nonumber\\
        &\leq& \left|\frac{\varepsilon}{4l}+d_L(y_i,p_{kg})+2\pi\bar{r}+\frac{\varepsilon}{4l}+d_L(y_j,p_{kd})\right.\nonumber\\ & &\left.-d_L(y_i,x_k)-d_L(x_k,y_j)\right|\nonumber\\
        &=&\frac{\varepsilon}{2l}+2\pi\bar{r}-2\bar{r}\nonumber\\
        &\leq&\frac{\varepsilon}{2l}+2\pi\bar{r}\label{desiq2}
        \end{eqnarray}
        is satisfied using (\ref{eq_ti}), (\ref{eq_e}) and the definition of the distance between points in $L$. Combining (\ref{desiq}) and (\ref{desiq2}) we get in this case
        \begin{eqnarray*}
            \delta&\leq&\left|d_S(S_i(\bar{t}),S_j(\bar{t}))-d_S(z_i,z_j)\right|+\left|d_S(S_i(\bar{t}),S_j(\bar{t}))-d_L(y_i,y_j)\right|\nonumber\\&<&2\pi\bar{t}+2\pi\bar{r}+\frac{\varepsilon}{2l}<\frac{\varepsilon}{4l}+\frac{\varepsilon}{4l}+\frac{\varepsilon}{2l}=\frac{\varepsilon}{l}.
        \end{eqnarray*}
 
 In the general case, when $y_i$ and $y_j$ are any points, we join them a geodesic in $L$ passing consecutively through the points $y_i=y_{i_0}$, $y_{i_1}$, $\dots$, $y_{i_k}=y_j$. We use the simple cases to get
 \begin{equation}\label{desiq3}
 \sum_{v=0}^{k}\left|d_S(z_{i_{v+1}},z_{i_{v}})-d_L(y_{i_{v+1}},y_{i_{v}})\right|<\varepsilon.
 \end{equation}
 Also, using a geodesic in $S$ joining $z_i$ and $z_j$, and passing consecutively through $S_i(\bar{t})=S_{i_0}(\bar{t})$, $\dots$, $S_{i_k}(\bar{t})=S_{j}(\bar{t})$,  we obtain
 \begin{equation}\label{desiq4}
     \sum_{v=0}^{d}\left|d_S(S_{i_{v+1}}(\bar{t}),S_{i_{v}}(\bar{t}))-d_L(y_{i_{v+1}},y_{i_{v}})\right|<\varepsilon.
 \end{equation}
 Joining (\ref{desiq3}) and (\ref{desiq4}),
 \begin{eqnarray*}
 \delta&=&\left|\sum_{v=0}^{k}d_S(S_{i_{v+1}}(\bar{t}),S_{i_{v}}(\bar{t}))-\sum_{v=0}^{k}d_L(y_{i_{v+1}},y_{i_{v}})\right|    \\
 &\leq&\sum_{v=0}^{k}\left|d_S(S_{i_{v+1}}(\bar{t}),S_{i_{v}}(\bar{t}))-d_L(y_{i_{v+1}},y_{i_{v}})\right|\\
 &\leq&\sum_{v=0}^{k}\left| d_S(z_{i_{v+1}},z_{i_{v}})-d_L(y_{i_{v+1}},y_{i_{v}})\right|<\varepsilon.
 \end{eqnarray*}

 Finally, by Lemma \ref{lema2}, $d_{\GH}(S,L)<6\varepsilon$. 
 
 We call P to the restriction of $S$ to the $T_i$'s and $C_{jk}$'s coming from $K$ and we notice that it is a metric space using $d_S$. If we take care of the connectivity of $K$ and work on every component of $K$ and $P$ as before, we obtain that $d_{\GH}(K,P)<6\varepsilon$. Thus, $d_{\GH}((S,P),(L,K))<12\varepsilon$.

Thanks to (\ref{dist_sing}), using sufficiently small collar neighborhoods, we can attach the $T_{i}$'s and $C_{jk}$'s in a smooth way keeping the estimates
\[
\delta<2\varepsilon\text{ and }\{z_i\}_{i=1}^{l}\text{ is a $12\varepsilon$-dense set.}
\]
From this we obtain
\[
d_{\GH}((S,P),(L,K))<24\varepsilon.
\]

\subsubsection*{Conclusion}

We estimate the distance the metric pairs $(X,A)$ and $(S,P)$ using (\ref{coro}) and the last results:
\begin{eqnarray*}
d_{\GH}((X,A),(S,P))&\leq& d_{\GH}((X,A),(L,K))+d_{\GH}((L,K),(S,P))\\
&\leq& \left(2^{\mu(n)}-1\right)\diam(X)+5^{-n}+24\varepsilon\\
&\leq&\left(2^{\mu(n)}-1\right)\diam(X)+5^{-n}+24\,(10)^{-(n+1)}\to 0\text{ as }n\to\infty.
\end{eqnarray*}
 \end{proof}
\begin{rema}
    If the subset $A$ of the metric pair $(X,A)$ is geodesically convex, then the 2-dimensional  submanifold given by theorem \ref{thm:cassorla} is connected.
\end{rema}

\begin{rema}
    For any compact metric tuple $(X,X_k,\dots,X_1)$ we can obtain a Riemannian surface along with $k$ nested 2-submanifolds with boundary. We follow the same construction as before, starting to build the 1-simplicial complex with the smallest closed subset.
\end{rema}

\section{Applications}\label{s:applications}

In recent years the Hausdorff and Gromov--Hausdorff distances have been used as quantitative measurements of the stability of different constructions in computational mathematics. 
In this section, we present some applied scenarios where the Gromov--Hausdorff distance for metric pairs and tuples appears and give a new perspective for those particular contexts. We believe that this section broadens the scope for the use of the Gromov--Hausdorff distance for metric pairs and tuples and presents interesting lines for future research.

\subsection{Hypernetworks}
\mbox{}%
\vspace{5pt}

Recently, Needham and collaborators delved into the theory of \textit{hypergraphs} and \textit{hypernetworks} \cite{needham_original, needham2024}. A \textit{network} is a triple $(X,X,\omega)$, where $X$ is a set and $\omega\colon X\times X\to \RR $ is an arbitrary function. A \textit{hypergraph} is a pair $(X,Y)$ where $X$ represents the set of nodes and $Y$ is a collection of subsets of $X$, called hyperedges. More generally, an \textit{hypernetwork} is a triple $(X,Y,\omega)$ such that $X$ and $Y$ are sets and $\omega\colon X\times Y\to\mathbb{R}$ is an arbitrary function. In particular, hypernetworks generalize networks, which arise as the special case $Y=X$ (see Example 2.8 of \cite{needham2024}). A hypergraph is an example of a hypernetwork where $\omega$ is the incidence function.

We can see a metric pair $(X,A)$ as a network $(X\times A,X\times A,\omega_X)$ with 
\[
\omega_X\colon(X\times A)\times(X\times A)\to \RR
\]
given by 
\[
\omega_X((x,a),(x',a'))=\frac{d_X(x,x')+d_X(a,a')}{2}.
\]
We can do something analogous for a metric tuple $(X,A_1,\dots,A_k)$ by considering 
\[
\omega_X\colon(X\times A_1\times\dots\times A_k)\times(X\times A_1\times\dots\times A_k)\to \RR
\]
given by 
\[
\omega_X((x,a_1,\dots,a_k),(x',a'_1,\dots,a'_k))=\frac{d_X(x,x')+d_X(a_1,a'_1)+\dots+d_X(a_k,a'_k)}{k}.
\]

In \cite{needham2024}, a Gromov--Hausdorff distance for networks and hypernetwork is defined. In the context of metric pairs seen as networks, taking the metric pairs $(X,A)$ and $(Y,B)$ seen as networks, and a correspondence $R$ between $X\times A$ and $Y\times B$ we get:
\begin{eqnarray*}
    \dis_{\net}(R)&:=&\sup\limits_{\substack{((x,a),(y,b)),\\ ((x',a'),(y',b'))\in R}}\left| \omega_X((x,a),(x',a'))-\omega_Y((y,b),(y',b')) \right|\\
    &=&\sup\limits_{\substack{((x,a),(y,b)),\\ ((x',a'),(y',b'))\in R}}\left| \frac{d_X(x,x')+d_X(a,a')-d_Y(y,y')-d_Y(b,b')}{2} \right|\\
    &\leq&\frac{1}{2}\left(\sup\limits_{\substack{((x,a),(y,b)),\\ ((x',a'),(y',b'))\in R}}\left| d_X(x,x')-d_Y(y,y')\right|+\sup\limits_{\substack{((x,a),(y,b)),\\ ((x',a'),(y',b'))\in R}}\left| d_X(a,a')-d_Y(b,b')\right| \right)\\
    &=&\dis(R).
\end{eqnarray*}
This means that 
\[
d_{\net}(X\times A, Y \times B):=\frac{1}{2}\inf\dis_{\net}(R)\leq d_{\GH}((X,A),(Y,B)),
\]
where the infimum runs over all correspondences $R$ between $X\times A$ and $Y \times B$.

\subsection{Simplicial Hausdorff Distance}
\mbox{}%
\vspace{5pt}

In \cite{nnadi}, Nnadi and Isaksen defined a version of the Hausdorff distance between simplicial complexes. To do that, they consider the class $\mathbb{X}^d$ of pairs $(X, f)$, where $X$ is a finite simplicial complex, $f\colon X_0\to\mathbb{R}^d$ is an injective map and $X_0$ is the set of vertices of $X$, and defined the distance as follows:

\begin{definition}[Definition 2.2, \cite{nnadi}]{\label{def:nnadi_hausdorff_distance}}
    The \textit{simplicial Hausdorff distance} is a map $\delta\colon\mathbb{X}^d\times\mathbb{X}^d\to\mathbb{R}_{\geq0}$, such that \begin{equation}{\label{eq:nnadi_hausdorff_distance}}
        \delta\left((X,f), (Y,g)\right) = \max\left\{\Vec{d}((X,f),(Y,g)), \Vec{d}((Y,g),(X,f))\right\},
    \end{equation}where the \textit{directed distance} $\Vec{d}((X,f),(Y,g))=\inf\{\epsilon>0\colon(X,f)\text{ is }\epsilon\text{-close to }(Y,g)\}$. The pair $(X,f)$ is \textit{$\epsilon$-close} to $(Y,g)$ if and only if: for every $k$-simplex $\sigma$ in $X$, there is a $k$-simplex $\tau$ in $Y$ such that for any vertex $v\in\sigma$, there is some vertex $w\in\tau$ with $d(f(v),g(w))<\epsilon$, for every $k$.
\end{definition}

Moreover, the authors also defined a distance between simplicial filtrations. Consider simplicial filtrations as sets $\{X_\alpha\}$ of simplicial complexes such that $X_\alpha\subseteq X_\beta$ for $\alpha\leq\beta$. Additionally, they also denote $(X,f)$ as a filtered complex, such that for every $\alpha>0$, $X_\alpha$ is a simplicial subcomplex on the vertex set $X_0$ with index value $\alpha$ and $f\colon X_0\to\mathbb{R}^d$ a measurement function on the vertex set. Then

\begin{definition}[Definition 2.20, \cite{nnadi}]{\label{def:nnadi_hausdorff_distance_filtered}}
    The \textit{simplicial Hausdorff distance} between two filtered complexes $(X,f)$ and $(Y,g)$ is \begin{equation}{\label{eq:nnadi_hausdorff_distance_filtered}}
        \hat{\delta}\left((X,f), (Y,g)\right)=\max\left\{\Vec{\mathrm{d}}((X,f),(Y,g)),\Vec{\mathrm{d}}((Y,g),(X,f))\right\},
    \end{equation}where $\Vec{\mathrm{d}}((X,f),(Y,g))=\sup_{\alpha>0}\{\Vec{d}((X_\alpha,f),(Y_\alpha,g))\}$.
\end{definition}

Let us denote by $|(X,f)|$ the geometric realization of the simplicial complex $X$ generated by the map $f$, i.e. the set given by
\[
\bigcup_{\sigma\in X} \Conv(f(\sigma))\subset \RR^D,
\]
where $\Conv(\cdot)$ denotes the convex hull, and $D\in\mathbb{N}$ big enough such that $|(X,f)|$ can be correctly embedded, i.e., $D\geq |X_0|$ would be enough. It is not difficult to see that 
\begin{equation}\label{eq:hausdorff-nnadi}
d_\H(|(X,f)|,|(Y,g)|)\leq \delta((X,f),(Y,g)),
\end{equation}
where $\delta$ denotes the simplicial Hausdorff distance from \cite{nnadi}.

Let $\{(X_\alpha,f)\}$ be a simplicial filtration. After taking geometric realizations, we obtain a filtration $\{|(X_\alpha,f)|\}$ with $|(X_\alpha,f)|\subset\mathbb{R}^D$, for every $\alpha$. Consider now $\{|(X_\alpha,f)|\}_{\alpha = 0}^k$ and $\{|(Y_\beta,g)|\}_{\beta = 0}^k$ the geometric realizations of two simplicial filtrations with the same length $k$. Then, we obtain metric tuples $(|(X_k,f)|,\dots,|(X_0,f)|)$ and $(|(Y_k,g)|,\dots,|(Y_0,g)|)$, and we can compute their Gromov--Hausdorff distance. By \eqref{eq:hausdorff-nnadi}, it follows that
\begin{equation}{\label{filtrationgromovhausdorff1}}
    d_{\H}(\{|(X_\alpha,f)|\},\{|(Y_\alpha,g)|\})=\sum_{\alpha=0}^{k}d_{\H}(|(X_\alpha,f)|,|(Y_\alpha,g)|)\leq (k+1)\hat\delta((X,f),(Y,g)),
\end{equation}
where $\hat{\delta}$ denotes the simplicial Hausdorff distance between filtered complexes, as defined in Definition \ref{def:nnadi_hausdorff_distance_filtered}. On the other hand, by definition,
\begin{equation}{\label{filtrationgromovhausdorff2}}
\hat{\delta}((X,f),(Y,g))\leq d_{\H}(\{|(X_\alpha,f)|\},\{|(Y_\alpha,g)|\}).
\end{equation}
Therefore, $\hat\delta$ and $d_\H$ are equivalent metrics on the set of simplicial complexes in $\RR^D$.

Furthermore, we observe that it is possible to globalize the distance between simplicial complexes introduced in \cite{nnadi} à la Gromov, as follows: consider pairs of the form $(X,\mathcal{X})$ where $X$ is a finite metric space and $\mathcal{X}$ is a simplicial complex with set of vertices $X$. Then, given two such pairs, $(X,\mathcal{X})$ and $(Y,\mathcal{Y})$, we can define their \textit{simplicial Gromov--Hausdorff distance} by
\[
\delta_{\GH}((X,\mathcal{X}),(Y,\mathcal{Y})) = \inf\{\delta^Z((\calX,\phi),(\calY,\psi))\}
\]
where the infimum runs over metric spaces $(Z,d_Z)$ and isometric embeddings $\phi\colon X\to Z$, $\psi\colon Y\to Z$, and where $\delta^Z$ is analogous to the simplicial Hausdorff distance, but for simplicial complexes embedded in $Z$, i.e.
\[
\delta^Z((\mathcal{X},\phi),(\mathcal{Y},\psi)) = \max\left\{\vec{d}_Z((\mathcal{X},\phi),(\mathcal{Y},\psi)),\vec{d}_Z((\mathcal{Y},\psi),(\mathcal{X},\phi))\right\}
\]
where
\[
\vec{d}_Z((\mathcal{X},\phi),(\mathcal{Y},\psi)) = \max_k \max_{\substack{\sigma \in \mathcal{X}\\
\dim(\sigma) = k}} \min_{\substack{\tau\in \mathcal{Y}\\
\dim(\tau) = k}} \max_{v\in \sigma} \min_{w\in \tau} d_Z(\phi(v),\psi(w)).
\]

It is obvious that $\delta_{\GH}$ is non-negative and symmetric. Moreover, the triangle inequality follows from the triangle inequality satisfied by the simplicial Hausdorff distance (which can be proved in any metric space $Z$ just as in $\RR^d$) and an argument analogous to the proof of \cite[Proposition 1.6]{jansen2017}. 

Finally, we say that $(X,\calX)$ and $(Y,\calY)$ are \textit{isomorphic} if there exists an isometry $\phi\colon X\to Y$ such that $\sigma$ is a $k$-simplex in $\calX$ if and only if $\phi(\sigma)$ is a $k$-simplicial complex in $\calY$. In that case, by taking $Z=Y$ and $\psi=\mathrm{id}_Y$, it easily follows that $\delta^Z((\calX,\phi),(\calY,\psi))=0$, which in turn implies $\delta_{\GH}((X,\calX),(Y,\calY))=0$. Conversely, if $\delta_{\GH}((X,\calX),(Y,\calY))=0$ then (following again the proof of \cite[Proposition 1.6]{jansen2017}), for any $n\in \NN$ there exist metric spaces $Z_n$ and isometric embeddings $\phi_n\colon X\to Z_n$, $\psi_n\colon Y\to Z_n$ such that 
\[
\delta^{Z_n}((\calX,\phi_n),(\calY,\psi_n))\leq \frac{1}{n}.
\]
By the finiteness of $\calX$ and $\calY$, up to passing to a subsequence, we can assume that for any $x\in X$ there exists $f(x)\in Y$ such that
\[
d^{Z_n}(\phi_n(x),\psi_n(f(x)))<\frac{1}{n}
\]
This induces a distance preserving map $f\colon X\to Y$. Then, again by the finiteness of $\calY$, we can assume that for any $k$-simplex $\sigma\in\calX$, $f(\sigma) \in \calY$. By repeating the same argument in the opposite direction, and up to passing to a further subsequence, we obtain a distance preserving map $g\colon Y\to X$ such that for any $k$-simplex $\tau \in \calY$, $g(\tau)\in \calX$ and such that 
\[
d^{Z_n}(\phi_n(g(y)),\psi_n(y))<\frac{1}{n}.
\]
Since $\phi_n$ and $\psi_n$ are isometric embeddings, we have
\[
d(g\circ f(x),x) \leq d^{Z_n}(\phi_n(g\circ f(x)),\psi_n(f(x)))+d^{Z_n}(\psi_n(f(x)),\phi_n(x)) \to 0.
\]
In other words, $f$ and $g$ induce isomorphisms between $(X,\calX)$ and $(Y,\calY)$.

The previous discussion shows that $\delta_{\GH}$ is truly a distance between finite metric simplicial complexes, and by \eqref{eq:hausdorff-nnadi} it easily follows that
\[
d_{\GH}(X,Y) \leq \delta_{\GH}((X,\calX),(Y,\calY)).
\]
We could even extend $\delta_{\GH}$ to simplicial filtrations: given simplicial filtrations $(X,\mathcal{X})=\{(X_\alpha,\mathcal{X}_\alpha)\}$ and $(Y,\mathcal{Y})=\{(Y_\alpha,\mathcal{Y}_\alpha)\}$, let
\[
\hat{\delta}_{\GH}((X,\mathcal{X}),(Y,\mathcal{Y})) = \max\left\{\vec{d}_Z((\mathcal{X},\phi),(\mathcal{Y},\psi)),\vec{d}_Z((\mathcal{Y},\psi),(\mathcal{X},\phi))\right\}
\]
where
\[
\vec{d}_Z((\mathcal{X},\phi),(\mathcal{Y},\psi)) = \sup_{\alpha}\vec{d}_Z((\mathcal{X}_\alpha,\phi),(\mathcal{Y}_\alpha,\psi)).
\]
By \eqref{filtrationgromovhausdorff1}
and \eqref{filtrationgromovhausdorff2}, the equivalence inequality follows:
\[
\hat\delta_{\GH}((X,\calX),(Y,\calY))\leq d_{\GH}(\{X_\alpha\},\{Y_\alpha\})\leq (k+1)\hat\delta_{\GH}((X,\calX),(Y,\calY)).
\]

\subsection{Persistence Matching Diagrams}
\mbox{}%
\vspace{5pt}

Persistent homology is a fundamental tool in Topological Data Analysis (TDA). Its application to filtrations of simplicial complexes yields barcodes or persistence diagrams, where intervals or points represent the birth and death of homological features across the filtration. In the following sections, we assume the reader is familiar with basic notions of TDA. For those less acquainted with the subject, we recommend the introductory references \cite{munch2017user, virk2022introduction}.

Formally, let $B(X) = (S_X, m(\cdot))$  be the barcode of certain simplicial filtration of the point cloud $X$, where $S_X$ is the set of intervals and $m(I)$ is the multiplicity of the interval $I$. Furthermore, the \textit{representation of a barcode} is defined as follows \[Rep(B(X))=\{(I,i)\colon I\in S_X\text{ and }i\leq m(I)\}.\]
If $B_1=(S_1,m_1(\cdot))$ and $B_2=(S_2,m_2(\cdot))$ are two barcodes, the \textit{partial matching} between $Rep(B_1)$ and $Rep(B_2)$ is a bijection $\sigma\colon R_1\subset Rep(B_1)\to R_2\subset Rep(B_2)$ and a \textit{block function} between $B_1$ and $B_2$ is a function $\mathcal{M}_f\colon S_1\times S_2\to \mathbb{Z}_{\geq 0}\cup\{\infty\}$ such that \[\sum_{J\in S_2}\mathcal{M}_f(I,J)\leq m_1(I), \text{ and }\sum_{I\in S_1}\mathcal{M}_f(I,J)\leq m_2(J).\] $\mathcal{M}_f$ induced a partial matching between $Rep(B_1)$ and $Rep(B_2)$.

González-Díaz, Soriano-Trigueros and Torras-Casas presented in \cite{torras_rocio_2} a particular block function $\mathcal{M}_f$ to obtain partial matchings between persistence diagrams induced by a morphism $f$. In \cite{torrascasas2024stability}, González-Díaz and Torras-Casas gave stability results for the persistence matching diagrams based on a block function $\mathcal{M}_f^0$, where the morphism $f$ is the inclusion of $X\subseteq Z$ and they pay attention to the $0$--homology barcode.

For that purpose, they defined a Gromov--Hausdorff distance for metric pairs in as follows \[\tilde{d}_{\GH}((Z,X),(Z',X')):=\inf \left\{\max\left\{d_{\H}^{d_{M}}(\gamma_{Z}(Z),\gamma_{Z'}(Z')),d_{\H}^{d_{M}}(\gamma_{Z}(X),\gamma_{Z'}(X'))\right\} \right\},\]
where the infimum runs over all metric spaces $(M,d_M)$ and isometric embeddings $\gamma_Z\colon Z\to M$ and $\gamma_Z'\colon Z'\to M$.


It holds that
\[
\tilde{d}_{\GH}((Z,X),(Z',X'))\leq d_{\GH}((Z,X),(Z',X'))\leq 2\ \tilde{d}_{\GH}((Z,X),(Z',X'))
\]
illustrating that the two metrics differ at most by a factor of two. As a result, any stability argument that relies on one remains valid, up to constant scaling, when using the other.

\nocite{*} 
\printbibliography

\end{document}